\documentclass[a4,11pt]{article}
\usepackage{hyperref}
\usepackage{amsmath,amssymb,amsthm,tikz,bbm,cleveref}
\usepackage[textsize=tiny]{todonotes}

\newcommand{\Z}{\mathbb{Z}}
\newcommand{\R}{\mathbb{R}}
\newcommand{\N}{\mathbb{N}}
\renewcommand{\P}{\mathbb{P}}
\newcommand{\E}{\mathbb{E}}
\newcommand{\mc}[1]{\mathcal{#1}}
\newcommand{\xc}{X^\circ}
\newcommand{\pv}{\boldsymbol{p}_m}
\newcommand{\av}{\boldsymbol{a}_m}
\newcommand{\bs}[1]{\boldsymbol{#1}}

\newcommand{\cweak}{\overset{w}{\to}}

\newcommand{\Mod}[1]{\,\mathrm{mod}\,#1}

\newcommand{\bro}{\gamma}
\newtheorem{theorem}{Theorem}
\newtheorem{proposition}{Proposition}
\newtheorem{lemma}{Lemma}
\newtheorem{conj}{Conjecture}
\newtheorem{cor}{Corollary}

\theoremstyle{definition}
\newtheorem{exa}{Example}
\newtheorem{defn}{Definition}
\newtheorem{rem}{Remark}

\newcommand{\xqed}[1]{%
   \leavevmode\unskip\penalty9999 \hbox{}\nobreak\hfill
   \quad\hbox{\ensuremath{#1}}}
\newcommand{\Enddef}{\xqed{\blacktriangleleft}}
\newcommand{\greedy}{\sigma_{\textnormal{greedy}}}
\newcommand{\doubleblind}[1]{}
\renewcommand{\doubleblind}[1]{#1}

\title{Cyclic products and optimal traps\\ in cyclic birth and death chains}
\doubleblind{\author{Mark Holmes\footnote{School of Math.~and Stat., the University of Melbourne.  {\tt holmes.m@unimelb.edu.au}}, Alexander E. Holroyd\footnote{School of Mathematics,  University of Bristol.  {\tt a.e.holroyd@bristol.ac.uk}}, and Alejandro Ram\'\i rez\footnote{Facultad de Matem\'aticas, Universidad Cat\'olica de Chile.  {\tt aramirez@mat.uc.cl}}}}

\begin{document}
\maketitle
\abstract{A birth-death chain is a discrete-time Markov chain on the integers whose transition probabilities $p_{i,j}$ are non-zero if and only if $|i-j|=1$.  We consider birth-death chains whose birth probabilities $p_{i,i+1}$ form a periodic sequence, so that $p_{i,i+1}=p_{i \mod m}$ for some $m$ and $p_0,\ldots,p_{m-1}$.  The trajectory $(X_n)_{n=0,1,\ldots}$ of such a chain satisfies a strong law of large numbers and a central limit theorem.  We study the effect of reordering the probabilities $p_0,\ldots,p_{m-1}$ on the velocity $v=\lim_{n\to\infty} X_n/n$.  The sign of $v$ is not affected by reordering, but its magnitude in general is.  We show that for Lebesgue almost every choice of $(p_0,\ldots,p_{m-1})$, exactly $(m-1)!/2$ distinct speeds can be obtained by reordering.  We make an explicit conjecture of the ordering that minimises the speed, and prove it for all $m\leq 7$.  This conjecture is implied by a purely combinatorial conjecture that we think is of independent interest. }

\medskip

\small{

}

\section{Introduction and main results}
Birth and death chains are (discrete-time, time-homogeneous) Markov chains on $\Z$ with transition probabilities $(p_{i,j})_{i,j\in \Z}$ satisfying $p_{i,i+1}+p_{i,i-1}=1$ for each $i \in \Z$.  Often ``birth and death chain'' allows $p_{i,i}>0$, but here for simplicity we assume that $p_{i,i}=0$.  

In this paper we consider \emph{cyclic}  birth and death chains  $X=(X_n)_{n \ge 0}$ on $\Z$, by which we mean that there exist $m\in \N$ and $\pv=(p_i)_{i=0}^{m-1}\in (0,1)^m$ such that for each $i \in \Z$, ($p_{i,i}=0$ and) $p_{i,i+1}=p_{i \Mod m}$.  Such models have been studied in general dimensions in e.g.~\cite{veto,take,KU}, and have been called \emph{random walks in periodic environment}.  To avoid any confusion with periodicity of a Markov chain we will refer to them as random walks in \emph{cyclic}  environment, or \emph{cyclic birth and death chains} (CBD).  In the 1-dimensional setting there is an elementary criterion for transience ($|X_n|\to \infty$) and recurrence ($X_n=0$ infinitely often), in which the crucial quantity is
\[\bro=\bro(\pv)=
\prod_{i=0}^{m-1}\rho_i,\quad \text{ where }\rho_i=\frac{1-p_i}{p_i}.\] 
The following result that can be proved using standard Markov chain  techniques.  Each conclusion holds with probability 1.
\begin{proposition}
\label{prp:sign}
Let $X$ be a CBD  with  $\pv:=(p_i)_{i=0}^{m-1}\in (0,1)^m$. Then $v(\pv):=\lim_{n\to \infty}n^{-1}X_n$ almost surely exists and is deterministic, and 
\begin{itemize}
\item $v(\pv)>0$ iff $\gamma<1$
\item $v(\pv)<0$ iff $\gamma>1$
\item $v(\pv)=0$ and $X$ is recurrent iff $\gamma=1$.
\end{itemize}
\end{proposition}
\begin{proof}
Observe the chain $X$ first at time 0, and thereafter observe the chain $X$ at times at which its displacement is $\pm m$ from the previous observation.  This new walk is ($m \times$) a simple random walk that is symmetric (hence recurrent, with velocity 0) in the third case above and biased to the right or left otherwise (see Lemma \ref{lem:Tspeed} and its proof below for more details).  Since the expected time for $X$ to reach $\pm m$ is finite, this proves the claim for the original chain $X$ as well.
\end{proof}

Motivated by trapping behaviour prevalent in \emph{random walk in random environment} on $\Z$ (where $(p_{i,i+1})_{i\in \Z}$ are chosen to be i.i.d.~random variables),
we are interested in how the velocity $v$ depends on the order of the $p_i$ for fixed $m$.  According to Proposition \ref{prp:sign}, the \emph{sign of $v$} (or equivalently, whether or not $X_n \to \pm \infty$) \emph{does not depend on the order of the $p_i$}.      If the velocity is 0 then it can't be changed by changing the order of the $p_i$, but in this case the variance may be of interest. Therefore we are primarily interested in the case where $X_n\to \infty$ (and $v>0$) with probability 1.  
  In particular, given a sequence $\pv=(p_i)_{i=0}^{m-1}\in (0,1)^m$ for which $\bro(\pv ) <1$, (so all velocities  arising from permutations will have positive sign), here are two natural questions that one can ask:
 
 \begin{itemize}
  \item[\bf{Q1:}]  What is the number $N(\pv)$ of distinct speeds achievable via permutations of $\pv$? 
 \item[\bf{Q2:}] In which order one should arrange these values to achieve the \emph{minimum} speed, or indeed the \emph{maximum} speed?
 \end{itemize}
  These questions do not seem to have been addressed in the literature previously, and they both turn out to be interesting.  In this paper we state some conjectures and provide partial answers to these questions, with our main results being Theorems \ref{thm:ae} and \ref{thm:greedy3} below.  There is trivially only 1 possible speed when $m=1,2$.  Theorem \ref{thm:ae} states that for $m\ge 3$ and Lebesgue a.e.~$\pv$ the answer to {\bf Q1} is $(m-1)!/2$.  This value arises from the fact  that the velocity is typically only invariant to rotations and reversal of the elements of $\pv$.   Note that invariance under rotations is trivial, while invariance under reversal seems to be a new (and we think surprising) result.  
\begin{theorem}
\label{thm:ae}
For any $m\ge 3$ and for Lebesgue a.e. $\pv\in (0,1)^m$ 
the number of distinct speeds satisfies 
\[N(\pv)= \frac{m!}{2m}=\frac{(m-1)!}{2}.\]  
Moreover, $N(\pv)\le (m-1)!/2$ for every $m\ge 3$ and $\pv\in (0,1)^m$.
\end{theorem}
Obviously the equality in Theorem \ref{thm:ae} cannot be satisfied (for $m>3$) for every $\pv=(p_0,\dots, p_{m-1})\in (0,1)^m$, since e.g.~if $p_i=p_j$ then the permutation that simply  switches $i$ and $j$ also preserves the speed. 
Theorem \ref{thm:ae} immediately implies that  when $m=3$ all rearrangements of $\pv$ give the same velocity, while for $m\ge 4$ and typical $\pv$, multiple different velocities are achievable via rearrangement.
 
To simplify discussions about ``optimal'' permutations, it is convenient (and loses no generality) to restrict attention henceforth to  $\pv$ for which the elements are non-increasing (so $p_0\ge p_1\ge \dots\ge p_{m-1}$).   In this case we believe that for fixed $m$ there exists a permutation $\sigma_{\text{greedy}}$ that is the universal minimiser of the speed for all such $\pv$.  
That is, for each $m$ there is a unique (up to rotations and reversals) permutation that minimises the speed no matter what the values of the $p_0\ge p_1\ge \dots\ge p_{m-1}$.    
 \begin{defn}
\label{def:cco}
Given a vector $\av=(a_0,a_1,\dots, a_{m-1})\in (0,\infty)^m$ with non-increasing entries, define the \emph{circular symmetric ordering} to be
 \[(a_0,a_2,a_4,\dots,a_5,a_3,a_1).\qquad \qquad \Enddef\]
 \end{defn}
 Let $\mc{S}_m$ be the set of permutations of $(0,1,\dots, m-1)$.  Of course,  $\mc{S}_m$ can be considered as the set of bijections from $\{0,1,\dots, m-1\}$ to itself.  We will use standard $()$  notation for permutations, e.g.~if $m=4$ and $\sigma=(0231)$ then $\sigma(0)=0$, $\sigma(1)=2$.  For a vector $\bs{x}=(x_i)_{i=0}^{m-1}$ and $\sigma\in \mc{S}_m$, write $\bs{x}_\sigma$ for the vector $(x_{\sigma^{-1}(i)})_{i=0}^{m-1}$.  For example, if $m=4$ and $\sigma=(0231)$ then $\sigma(0)=0$, $\sigma(1)=2$ etc., and $\bs{x}_\sigma=(x_0,x_3,x_1,x_2)$.    

We call the permutation corresponding to the circular symmetric ordering  $\greedy$, because it groups large values of $\av$ with  each other, and small values of $\av$ with  each other in a circular way.
  \begin{defn}
The \emph{greedy permutation} $\greedy$ is given by
\[\greedy^{-1}(i)=2i \quad \text{ and }\quad \greedy^{-1}(m-1-i)=2i+1,\]
for $i=0,1,\dots, \lfloor m/2\rfloor-1$.  \Enddef
\end{defn}
By definition, the greedy permutation depends on $m$ but not on the actual values $\av$, e.g.~if $m=9$ then $\greedy=(081726354)$.  Let $\av\in (0,\infty)^m$ with decreasing entries.  For $r\in [m]$ and a permutation $\sigma\in \mc{S}_m$ define 
\[P_r(\sigma;\av)=\sum_{k=0}^{m-1}\prod_{i=0}^{r-1}a_{\sigma^{-1}(k+i)},\]
with indices interpreted $\Mod m$.  

The following conjecture says that the greedy permutation maximises $P_r$ for each $r$.   We think that it is an interesting standalone open problem.  It also immediately implies that the greedy permutation minimises the speed (see Proposition \ref{prp:explicit} and Conjecture \ref{conj:min} below).
\begin{conj}[Greedy cyclic products are maximal]
\label{conj:greedy}
Let $\av\in (0,\infty)^m$ have decreasing entries.  Then for every  $r\in [m]$, and every $\sigma\in \mc{S}_m$,
\[P_r(\greedy; \av)\ge P_r(\sigma; \av).\]
\end{conj}
The cases $r=1$ and $r=m$ in Conjecture \ref{conj:greedy} are trivially true.  The case $r=2$ is not difficult to prove, and appears as early as \cite{Abram0}.  Such facts are termed \emph{circular rearrangement inequalities} in \cite{Yu} (see also \cite{Abram}).  We present a simple proof for the case $r=2$ and also give a (non-trivial) proof in the case $r=3$.  Observe that 
\[\dfrac{\prod_{i=0}^{r-1} b_j}{\prod_{i=0}^{m-1} b_j}=\prod_{i=r}^{m} \frac{1}{b_j},\]
from which we conclude that if $a$ is the product of all elements of $\av$ then 
\begin{equation}
\label{trick1}
a^{-1}P_r(\sigma;\av)=P_{m-r}(\sigma;\av^{-1}),
\end{equation}
where $\av^{-1}=(a_0^{-1},\dots, a_{m-1}^{-1})$.   This observation together with the aforementioned results for $r\le 3$ gives rise to the following.
\begin{theorem}
\label{thm:greedy3}
The conclusion of Conjecture \ref{conj:greedy} holds for $(m,r)$ such that $r\le 3$ or $m-3\le r\le m$.
\end{theorem}
It is immediate from Theorem \ref{thm:greedy3} that the conclusion of Conjecture \ref{conj:greedy} holds for all $r\in [m]$ when $m\le 7$.

The relevance of Conjecture \ref{conj:greedy} and Theorem \ref{thm:greedy3} to {\bf Q2} can be seen from the following explicit formula for the velocity, in which $I$ denotes the identity permutation and $\bs{\rho}_m=(\rho_0,\rho_1,\dots, \rho_{m-1})$.  

\begin{proposition}
\label{prp:explicit}
For CBD with transition probabilities $\pv\in (0,1)^m$ such that $\bro<1$ we have 
\begin{equation}
v(\pv)=\dfrac{1-\bro}{1-\bro+\frac{2}{m}\sum_{r=1}^{m}P_{r}(I,\bs{\rho}_m)}.\label{awesome}
\end{equation}
\end{proposition}
Since for any $r$, sums over starting indices $k$ of consecutive products of $\rho_\cdot$ are invariant under rotations and reversals (reversing the order of $\pv$), we can immediately conclude from \eqref{awesome} that $v(\pv)$ is invariant under rotations and reversals of the elements of $\pv$, as claimed earlier.  In particular, when $m=3$ there is only one possible velocity, since all permutations are combinations of

As noted earlier, we find the fact that the speed is invariant under reversals to be  somewhat surprising, and is not at all obvious from other expressions for the velocity.   For example, as in Lemma \ref{lem:circ} in Section \ref{sec:speed_formulae} below, the velocity can also be written as $v(\pv)=\sum_{i=0}^{m-1}\pi_i (2p_i-1)$, where $\bs{\pi}=(\pi_0,\dots, \pi_{m-1})$ is the stationary distribution of the chain $\xc_n=X_n \Mod m$.  This stationary distribution behaves ``nicely'' under rotations but \emph{not} under reversal of the elements of $\pv$- see Example \ref{exa:m3}  in Section \ref{sec:speed_formulae} below.   

Given $\pv\in (0,1)^m$ with non-increasing entries, and $\sigma\in \mc{S}_m$, let $v(\sigma;\pv)=v((\pv)_{\sigma})$.  As noted above, the following is an immediate corollary of Conjecture \ref{conj:greedy} (and Proposition \ref{prp:explicit}).  It says that the greedy permutation minimises the speed.  
\begin{conj}[Greedy is least speedy]
\label{conj:min}
For any $\pv$ with non-increasing entries, such that $\bro<1$ (so all possible speeds will be positive), then for every $\sigma\in \mc{S}_m$,
\[ v(\greedy; \pv)\le v(\greedy; \pv)\] 
\end{conj}
For example, if $m=9$ and the $p_i$ are decreasing in $i$ with $\rho(\pv)<1$ then according to Conjecture \ref{conj:min}, for any permutation $\sigma\in \mc{S}_m$,
\[ v\big((p_0,p_2,p_4,p_6,p_8,p_7,p_5,p_3,p_1)\big)\le v((\pv)_\sigma).\]

As a consequence of Theorems \ref{thm:ae} and \ref{thm:greedy3} above we obtain the following (we omit the proof).
\begin{cor}
\label{cor:greedy6}
For $m\le 7$ and $\pv$ with non-increasing entries and $\rho<1$, the speed is minimised by the greedy permutation (i.e.~Conjecture \ref{conj:min} holds for $m\le 7$). 
\end{cor}

\begin{rem}
Notice that for $m\le 7$  the ordering that minimises the speed is also the one that \emph{minimises}  (interpret the following with indices mod $m$),
\[\mc{L}(\pv):=\sum_{i=0}^{m-1} (p_{i+1}-p_{i})^2\]
 (expand the square, and note that only the sum of mixed terms depends on the order).  One might interpret this as saying that the speed is minimised by having a ``smooth'' ordering (a cyclic arrangement of the elements of $\pv$ that has no large jumps). \Enddef
\end{rem}

We stress that the maximiser of $P_r$ appears to be universal.  In other words we believe that the greedy permutation maximises $P_r$ for every $r$ and every $\av$ with decreasing entries.   In the language of \cite{Yu} this says that the \emph{circular symmetrical order} maximises $P_r$.  In \cite{Yu} it was shown that the so-called \emph{circular alternating order} minimises $P_2$.  
The following two examples (which can be verified by simply evaluating the cyclic products for all possible permutations) show that the minimal ordering is neither constant over $r$ for fixed $\av$, nor constant over $\av$ for fixed $r$. 
\begin{exa}
For the vector $\av=(9,7,6,5,4,3)$, the minimal value of $P_2(I; )$ is achieved by the ordering $(9,3,7,5,6,4)$ (and not by the ordering $(9 , 3,  6,  7 ,  4  , 5)$) while the minimal value of $P_3(I; )$ is achieved by the ordering $(9 , 3,  6,  7 ,  4  , 5)$ (and not by  $(9,3,7,5,6,4)$).\Enddef
\end{exa}
\begin{exa}
For the vector $\av=(10,5,4,3,2,1)$, the minimal value of $P_3(I; )$ is achieved by the ordering $(10,1,\bs 4,\bs 5,2,3)$ (and not by the ordering $(10,1,\bs 5,\bs 4,2,3)$). For the vector $\av'=(10,9,6,5,3,1)$ the minimal value of $P_3(I;)$ is achieved by the ordering $(10,1,\bs 9, \bs 6,3,5)$ (and not by  $(10,1,\bs 6,\bs 9,2,3)$).
\end{exa}
Since the permutations which minimise products of $r$ consecutive terms in the cycle are not constant over $r$, The above observations don't give any conclusion for permutations that maximise the speed.  Nevertheless, we have the following.
\begin{exa} 
\label{exa:max_speeds}$\,$
\begin{itemize}
\item[(i)] If $\pv=(9,8,7,6,5,4,3)/10$ then the maximum value of the speed $v((\pv)_\sigma)$ is obtained by
$v\big((6,7,4,9,3,8,5)/10\big)\approx 0.19857$, corresponding to the permutation $\sigma=(35\bs{10}624)$.
\item[(ii)] If $\pv'=(8,7,6,5,4,3.5,3)/10$  then the maximum value of the speed $v((\pv')_{\sigma'})$ is obtained by 
$v\big((6,5,3.5,8,3,7,4)/10\big)\approx 0.04675$, corresponding to the permutation $\sigma'=(35\bs{01}624)$.
\end{itemize}
Moreover, in case (i) above the speed is not maximised at $v((\pv)_{\sigma'})\approx 0.19787$, and in case (ii) above the speed is not maximised at $v((\pv')_{\sigma})\approx 0.04668$.\Enddef
\end{exa}

The remainder of this paper is organised as follows:  In Section \ref{sec:discussion} we further discuss the context of our results: we present some elementary (implicit) speed formulae, compare results about the velocity for the cyclic birth and death chain to that of a related model of random walk in random environment, and briefly discuss the central limit theorem.  In Section \ref{sec:awesome_formula} we prove Proposition \ref{prp:explicit}.  To do this we follow an approach that will be familiar to researchers in the area of \emph{Random Walk in Random Environment} (RWRE), and then manipulate the resulting expression to get \eqref{awesome}.  Some understanding of discrete-time Markov chains (specifically birth and death chains) is required to understand Sections \ref{sec:discussion} and  \ref{sec:awesome_formula}.  The reader who is happy to start with \eqref{awesome} as given can proceed directly to Sections \ref{sec:ae} and \ref{sec:r3} where 
we prove Theorems \ref{thm:ae} and \ref{thm:greedy3}   respectively.   

 Theorem \ref{thm:ae} will be proved by showing: (i) that the speed is invariant to rotations and reversal of the elements of $\pv$ (the former is trivial, while we find the latter to be rather surprising), and; (ii) for typical $\pv$ these kinds of permutations are the only ones which do not change the speed.  
 
 Theorem \ref{thm:greedy3} will be proved by induction on $m$ for $r=2,3$.  The cases $3<r<m-3$ remain open.


\section{Discussion}
\label{sec:discussion}
In this section we further discuss the context of our results.  Let us begin with some simple (and standard) implicit formulas for the velocity.
\subsection{Elementary speed formulae}
\label{sec:speed_formulae}
Let $\xc_n=X_n \Mod m$. Then $\xc=(\xc_n)_{n \ge 0}$ is also an irreducible discrete-time Markov chain (typically non-reversible), with finite state space $\{0,1,\dots, m-1\}$ and transition probabilities $p_{0,m-1}=1-p_0$, $p_{m-1,0}=p_{m-1}$, and $p_{i,i+1}=p_i$ for $i<m-1$ and $p_{i,i-1}=1-p_i$ for $i>1$. Let $\bs{\pi}=(\pi_i)_{i=0}^{m-1}$ denote the stationary distribution of $\xc$ (which depends on $\pv$). Then we have the following.
\begin{lemma}
\label{lem:circ}
$v(\pv)=\sum_{i=0}^{m-1}\pi_i (2p_i-1)$.
\end{lemma}
Readers familiar with random walk in random environment might interpret Lemma \ref{lem:circ} as a formula for the speed given in terms of the environment viewed from the particle.  One can find an explicit (albeit complicated) formula for $\bs{\pi}$, and hence for $v$, by solving a recursion for mean return times, but we will not present this here.   The following example however demonstrates that the invariance of the speed (under all rotations) in the case $m=3$ is not at all trivial.
\begin{exa}
\label{exa:m3}
For the case $m=3$, the stationary distribution satisfies (with subscripts interpreted $\Mod 3$) for $i=0,1,2$,
\[\pi_i(p_0,p_1,p_2)=\frac{p_{i+1}p_{i+2}-p_{i+1}+1}{p_0p_1+p_0p_2+p_1p_2-p_0-p_1-p_0+3}.\]
The denominator $d=d(\{p_0,p_1,p_2\})$ is invariant under permutations.  Note that e.g.
\[\pi_{0}(p_2,p_1,p_0)=\dfrac{p_{1}p_{0}-p_{1}+1}{d},\]
which is not equal to any of the $\pi_i(p_0,p_1,p_2)$ in general.  E.g.~$\bs{\pi}(0.4,0.6,0.8)=(22,13,21)/56$, while  $\bs{\pi}(0.8,0.6,0.4)=(16,23,17)/56$.  Nevertheless, $v(0.4,0.6,0.8)=v(0.8,0.6,0.4)=27/140$.
\end{exa}
\begin{proof}[Proof of Lemma \ref{lem:circ}]
Let $(\Delta_{j,i}:j=0,\dots, m-1, i\in \N)$ be independent random variables with \[\P(\Delta_{j,i}=1)=p_{j}=1-\P(\Delta_{j,i}=-1).\] 

For $j=0, 1, \dots, m-1$, and $n\ge 1$, let $N_n(j)=\#\{r<n: X_r \Mod m =j\}$. Then 
\begin{equation*}
X_n=\sum_{j=0}^{m-1} \sum_{i=1}^{N_n(j)}\Delta_{j,i}.
\end{equation*}
Thus,
\begin{equation*}
n^{-1}X_n=\sum_{j=0}^{m-1} \frac{N_n(j)}{n}\frac{1}{N_n(j)}\sum_{i=1}^{N_n(j)}\Delta_{j,i}.
\end{equation*}
Note that $N_n(j)$ is the number of visits by (the irreducible, finite-state DTMC) $\xc$ to $j$ prior to time $n$. Therefore $n^{-1}N_n(j)\to \pi_j$ almost surely. Since $\Delta_{j,i}$ are independent this implies that 
\[n^{-1}X_n\to \sum_{j=0}^{m-1} \pi_j \E[\Delta_{j,1}]=\sum_{j=0}^{m-1} \pi_j(2p_j-1).\]
\end{proof}
Let $T$ denote the first hitting time of $\{-m,m\}$ by the chain $X$, and let $h=\P(X_T=m)=1-\P(X_T=-m)$. Then a standard resistance calculation gives $h=(1+\bro)^{-1}$, 
and we have the following formula.
\begin{lemma}
\label{lem:Tspeed}
$v(\pv)=\dfrac{\E[X_T]}{\E[T]}=m \cdot \dfrac{2h-1}{\E[T]}$.
\end{lemma}
\begin{proof}
Let $T^{(0)}=0$, and for $i\ge 1$ let  $T^{(i)}=\inf\{k> T^{(i-1)}:X_k-X_{T^{(i-1)}}\in \{-m,m\}\}$.  Since $n^{-1}X_n \to v$ almost surely we have that $X_{T^{(i)}}/T^{(i)}\to v$ as $i \to \infty$.  By the law of large numbers, $i^{-1}T^{(i)}\to \E[T]$ and $i^{-1}X_{T^{(i)}}\to \E[X_T]=m(\P(X_{T}=m)-\P(X_{T}=-m))$.
\end{proof}  
Let $T_+$ denote the first hitting time of $m$ by the chain $X$. Then standard renewal  arguments give the following.
\begin{lemma}
\label{lem:T+speed}
If $\bro<1$ then 
$v(\pv)=\dfrac{m}{\E[T_+]}$.
\end{lemma}
\begin{proof}
Let $T_+^{(0)}=0$ and for $i\ge 1$ let  $T_+^{(i)}=\inf\{k> T_+^{(i-1)}:X_k-X_{T_+^{(i-1)}}=m\}$, which is finite almost surely since $\bro<1$.  Now proceed as in the proof of Lemma \ref{lem:Tspeed}.
\end{proof}
Each of the above representations for $v$ is standard, but we would describe as implicit in the sense that $\bs{\pi}$ in Lemma \ref{lem:circ} and the expectations in the denominators in Lemmas \ref{lem:Tspeed} and \ref{lem:T+speed} are not explicit functions of $\pv$.  Nevertheless, we will use Lemma \ref{lem:T+speed} to  prove Proposition \ref{prp:explicit}.  It is intuitively obvious that for $\bro<1$ the denominator in Lemma \ref{lem:T+speed} is strictly decreasing in each $p_i$.  This can be made rigorous via a simple coupling argument to obtain the following.
\begin{lemma}
\label{lem:strictmono}
$v(\pv)$ is strictly increasing in each $p_i\in (0,1)$.
\end{lemma} 
\begin{proof}
Let $\pv\in (0,1)^m$ be given.  Symmetry arguments allow us to assume without loss of generality that $\bro=\bro(\pv) \le 1$.  Let $\pv'$ be equal to $\pv$ except that $p'_i>p_i$.  If $\bro=1$ then the claim holds by Proposition \ref{prp:sign}.  Otherwise $\bro<1$ and $\E[T_+]<\infty$ in Lemma \ref{lem:T+speed}.  It is easy (see e.g.~\cite{Holmes1,Holmes2}) to define a probability space on which copies of the CBD $\pv$ and the CBD $\pv'$ are both defined, and such that: (i)  $T_+'\le T_+$ almost surely, and (ii) $T_+'< T_+$ with positive probability.  This shows that $\E[T'_+]<\E[T'_+]$ in Lemma \ref{lem:T+speed} which completes the proof.
\end{proof}




For $u\in [0,1)$, let $\mc{P}_m(u)=\{\pv\in (0,1)^m:v(\pv)=u\}$ denote the set of (ordered) vectors of length $m$ that have speed $u$.  According to  Lemma \ref{lem:strictmono}, for each $p_1,\dots, p_{m-1}$ there is at most one value of $p_0$ for which $v(\pv)=u$.  Therefore $\mc{P}_m(u)$ is a subspace of dimension at most $m-1$, and it has Lebesgue measure 0.  

\subsection{Comparison with RWRE}
\label{sec:RWRE}
If one adds a uniform random shift of the environment (shift the environment by $i \in \{0,1,\dots m-1\}$ with probability $1/m$ for each $i$), this model can be viewed as an example of a random walk in a (quenched) ergodic environment.  To be precise, given the vector of elements $\pv=(p_i)_{i=0}^{m-1}$ let $\Omega$ be the set of bi-infinite sequences $\bs{\omega}=(\omega_x)_{x \in \Z} $ taking values in $\{p_0,\dots, p_{m-1}\}$ for which there exists some $i\in 0,\dots, m-1$ such that $(\omega_j)_{j=0}^{m-1}=(p_{(j+i) \Mod m})_{j=0}^{m-1}$ and 
$\omega_{x}=\omega_{x \Mod m}$ for each $x\in \Z$.  There are at most $m$ distinct elements in $\Omega$.  Let $\mc{F}$ be the power set of $\Omega$, and $\mu$ be the uniform measure on $\Omega$.  Then $(\Omega, \mc{F},\mu)$ is ergodic with respect to the shift operator $\theta((\omega_x)_{x \in \Z})=(\omega_{x+1})_{x \in \Z}$  \footnote{Indeed any $B\in \mc{F}$ can be expressed as an event depending only on the state of $(\omega_0,\dots, \omega_{m-1})$, i.e.~$B=\{(\omega_0,\dots,\omega_{m-1})\in D\}$ for some $D$.  If $\mu(B)>0$ then there exists $i\in \{0,1,\dots, m-1\}$ such that $(p_{(i+j)\Mod m})_{j=0}^{m-1}\in B$, and therefore at least one of $\theta^{-k}(B)$ occurs, so $\mu(\cup_{n\in \N}\theta^{-n}(B))=1$.}.  As such, any result from the theory of random walk in ergodic random environment that holds for a.e.~environment holds for the CBD with $\omega_x=p_{x \Mod m}$ etc.  For example, a law of large numbers with an implicit formula for the speed, is known to hold for random walk in ergodic random environment, see e.g.~\cite{Zeitouni}.


It is natural to compare results for cyclic birth and death (CBD) processes to those for (uniformly elliptic) i.i.d.~RWRE with right step probability from each site being uniformly selected from our set of probabilities $\{p_0,p_1,\dots, p_{m-1}\}$ (counting multiplicites if there are any). The results of Solomon \cite{Solomon} in this special setting become: 
\begin{itemize}
\item The walker is transient to $+\infty$ if and only if 
\begin{align}
m^{-1}\sum_{i=0}^{m-1}\log(\rho_i)<0\iff \prod_{i=0}^{m-1}\rho_i<1.\label{transcrit}
\end{align}
\item If the walker is transient to $+\infty$ then the velocity is strictly positive if and only if $m^{-1}\sum_{i=0}^{m-1}\rho_i<1$, in which case the velocity is equal to 
\begin{equation*}
\dfrac{1-m^{-1}\sum_{i=0}^{m-1}\rho_i}{1+m^{-1}\sum_{i=0}^{m-1}\rho_i}.
\end{equation*}
\end{itemize}
In other words, the criteria for transience (for RW i.i.d.~RE and for  cyclic birth and death chains) ``match'', but the criteria for positivity of the speed do not. Both of these observations are to be expected - in the former case one can see the criteria as coming from a calculation involving the resistance to $+\infty$ (and $-\infty$) together with the LLN for the limiting proportion of time that each environment appears. In the latter case obviously the velocity of the RWRE above should be invariant to permutations of the elements of $\pv$ since choosing a uniform i.i.d.~sample from $\pv$ ignores any ordering.  Moreover, the disorder in the environment allows much stronger traps to be created. 
In view of the last observation, it is natural to ask whether the speed for this RWRE is always less than the CBD (when \eqref{transcrit} holds). This can be easily checked when $m=2$. Numerical examples (e.g.~$\pv=(0.57,0.87,0.98,0.79,0.64,0.56)$) suggest that this is not the case in general when $m>4$.   In other words we believe that for each $m>4$ there exist examples where the speed $v(\pv)$ of the CBD is strictly positive, but smaller than the speed of the corresponding RWRE.  We interpret this observation as saying that when $m$ is large it is possible to create really bad traps in CBD by very specific orderings of $\pv$ and that traps as bad or worse occur extremely rarely in the i.i.d.~RE.  When $m$ is small any particular ordering of the $\pv$ will appear fairly often in the i.i.d.~RE, as will ``even worse'' traps.

\subsection{CLT}
\label{sec:CLT}
Thus far we have only discussed how the deterministic limiting velocity behaves as a function of $\pv$. One might also ask about the variance, and a central limit theorem. Let $T_0=0$ and $T_k=\inf\{n>T_{k-1}:|X_n-X_{T_{k-1}}|=m\}$ and $W_k=m^{-1}X_{T_k}$. Then $(W_k)_{k \in \N}$ is a nearest-neighbour simple random walk on $\Z$ with $\P(W_k=W_{k-1}+1)=h$. It follows immediately that $k^{-1}W_k\to 2h-1$ almost surely. Moreover,
\begin{align*}
k^{-1/2}(W_k-k(2h-1))\cweak \mc{N}(0,4h(1-h)).
\end{align*}
We cannot apply the standard CLT for random walk in ergodic random environment (e.g.~\cite[Theorem 2.2.1]{Zeitouni}) because our environment is non-mixing (it is completely determined by its value in any interval of length $m$).  Nevertheless one can use the Markov chain central limit theorem to obtain a CLT (see e.g.~\cite{veto}):  For each $\pv\in (0,1)^m$ there exists a deterministic $\sigma^2=\sigma^2(\pv)>0$ such that
\begin{equation*}
\dfrac{X_n-nv}{\sqrt{n}}\cweak \mc{N}(0,\sigma^2).
\end{equation*} 
The constant $\sigma^2$ can be expressed in terms of $\bs{\pi}$ and $k$-step transition probabilities for all $k$, but is not really tractable in this form.  It would be of interest to find  a more explicit expression in terms of $\pv$.  In the case $v=0$, Takenami \cite{take} has proved a local limit theorem for the walk.

\section{Proof of Proposition  \ref{prp:explicit}}
\label{sec:awesome_formula}
Fix $m$, $\pv$, and recall Lemma \ref{lem:T+speed}.  For $i\ge 0$ let
    $$
S_i:=\min\{n\ge 0: X_n=i\}.
$$
Note that since the random walk is transient to the
right, we have that $S_i<\infty$ $a.s.$
We will derive a set of $m$ linear equations for
$\E_0[S_1], \E_1[S_2],\ldots, \E_{m-1}[S_m]$, where $\E_j$ denotes expectation with respect to the law of the chain $\bs{X}$, starting from state $j$.  Note that 
\begin{align*}
\E_i[S_{i+1}]=1+(1-p_i)\E_{i-1}[S_{i+1}]=1+(1-p_i)\big(\E_{i-1}[S_i]+\E_{i}[S_{i+1}]\big).
\end{align*}
Therefore $p_i\E_i[S_{i+1}]=1+(1-p_i)\E_{i-1}[S_i]$.  
This set of equations can be written as
$
M {\bf e}={\bf 1}
$
where
\begin{equation*}
  M:=
  \left[ {\begin{array}{cccccc}
   p_0 & 0 & 0 &\cdots & 0 & -(1-p_0) \\
            -(1-p_1)& p_1 &  0 &\cdots & 0 & 0 \\
            0 & -(1-p_2) &  p_2 &\cdots & 0& 0 \\
            . & . & . & \cdots & . & .\\
            . & . & . & \cdots & . & . \\
            . & . & . & \cdots & p_{m-2} & 0\\
            0 & 0 & 0& \cdots& -(1-p_{m-1}) & p_{m-1}
  \end{array} } \right],
\end{equation*}
${\bf e}:=(\E_0[S_1],\ldots, \E_{m-1}[S_m])$ and
${\bf 1}:=(1,\ldots, 1)$.
From Cramer's rule we get that for $i=0,\dots , m-1$,
\begin{equation*}
\E_i[S_{i+1}]=\frac{|M^{(i+1)}|}{|M|}
\end{equation*}
where $M^{(j)}$ is the matrix obtained after replacing the $j$-th
column of $M$ by ${\bf 1}$, and $|A|$ denotes the determinant of $A$.  Since  $\E[T_+]=\sum_{i=0}^{m-1}\E_i[S_{i+1}]$ we have from Lemma \ref{lem:T+speed} that
\begin{equation*}
v=m\cdot \frac{|M|}{\sum_{i=0}^{m-1}|M^{(i+1)}|}.
\end{equation*}

Now note that
\begin{align*}
|M|&=p_0p_1\cdots p_{m-1}-(1-p_0)(1-p_1)\cdots (1-p_{m-1})\\
&=\prod_{i=0}^{m-1}p_i-\prod_{i'=0}^{m-1}(1-p_{i'}),
\end{align*}
which is invariant under permutations
on the sub-indices $0,1,\ldots, m-1$. Also note that
\begin{align*}
|M^{(m)}|&=\prod_{j=0}^{m-2} p_j
+\prod_{j=0}^{m-3}p_j(1-p_{m-1})
  +\prod_{j=0}^{m-4}p_j(1-p_{m-2})(1-p_{m-1})\nonumber\\
  &\quad +\cdots+
  p_0\prod_{j=2}^{m-1}(1-p_{j})+\prod_{j=1}^{m-1}(1-p_j)\\
&=  \sum_{j=0}^{m-1}\prod_{i_1=0}^{j-1}p_{i_1}\prod_{i_2=j+1}^{m-1}(1-p_{i_2}),
\end{align*}
while the other $M^{(i+1)}$ are of the same form but with the indices rotated.   Let $\mc{R}_m$ be the set of rotation permutations, that is, compositions of the permutation $(123\dots 0)$.  It follows that the velocity can be written as 
\begin{align*}
v(\pv)
&=m \cdot \dfrac{\prod\limits_{i=0}^{m-1}p_i-\prod\limits_{i=0}^{m-1}(1-p_i)}{\sum\limits_{\sigma\in \mc{R}_m} \sum\limits_{j=0}^{m-1}\prod\limits_{i_1=0}^{j-1}p_{\sigma(i_1)}\prod\limits_{i_2=j+1}^{m-1}(1-p_{\sigma(i_2)})}\\
&= \dfrac{m\cdot \prod\limits_{i=0}^{m-1} p_i\cdot (1 -\rho)}
{\sum\limits_{k=0}^{m-1} \sum\limits_{j=0}^{m-1} (p_{j+k}+(1-p_{j+k}))\prod\limits_{i_1=0}^{j-1} p_{i_1+k} \prod\limits_{i_2=j+1}^{m-1}(1- p_{i_2+k})}.
\end{align*}
The denominator is equal to 
\[\sum_{k=0}^{m-1}\sum_{j=0}^{m-1}\prod_{i_1=0}^{j} p_{i_1+k} \prod_{i_2=j+1}^{m-1}(1- p_{i_2+k})+  \sum_{k=0}^{m-1}\sum_{j=0}^{m-1}\prod_{i_1=0}^{j-1} p_{i_1+k} \prod_{i_2=j}^{m-1}(1- p_{i_2+k}).\]
This can be written as
\begin{equation}
\prod_{i=0}^{m-1}p_i \cdot \Bigg[\sum_{j=0}^{m-1}\sum_{k=0}^{m-1} \prod_{i_2=j+1}^{m-1}\rho_{i_2+k}+\sum_{j=0}^{m-1}\sum_{k=0}^{m-1} \prod_{i_2=j}^{m-1}\rho_{i_2+k}\Bigg].\label{giraffe}
\end{equation}

Letting $r=m-j$ and $i=m-1-i_2$ and using the fact that the sum over $k$ is a sum over the whole cycle, we see that the second term in the square brackets in \eqref{giraffe} is equal to 
\[
\sum_{r=1}^m \sum_{k=0}^{m-1}\prod_{i=0}^{r-1}\rho_{i+k}.\]

By separating off the term $j=m-1$, and using the fact that an empty product is equal to 1, the first term in the square brackets in \eqref{giraffe} is 
\begin{align}
\sum_{j=0}^{m-2}\sum_{k=0}^{m-1} \prod_{i_2=j+1}^{m-1}\rho_{i_2+k}+  \sum_{k=0}^{m-1} 1=\sum_{j=0}^{m-2}\sum_{k=0}^{m-1} \prod_{i_2=j+1}^{m-1}\rho_{i_2+k}+  m.\label{goat1}
\end{align}
Now let $r=m-j-1$ and $i=m-1-i_2$ to see that this is equal to
\[\sum_{r=1}^{m-1}\sum_{k=0}^{m-1} \prod_{i=0}^{r-1}\rho_{i+k}+  m=\sum_{r=1}^{m}\sum_{k=0}^{m-1} \prod_{i=0}^{r-1}\rho_{i+k}-m\bro+  m
.\label{goat2}\]
It follows that

\[
v(\pv)
= m\cdot \dfrac{1-\bro}{m-m\bro+2\sum\limits_{r=1}^{m}\sum\limits_{k=0}^{m-1}\prod\limits_{i=0}^{r-1}\rho_{i+k}}.
\]    
Cancelling factors of $m$ and using the definition of $P_r$ completes the proof.\qed

\section{Proof of Theorem \ref{thm:ae}}
\label{sec:ae}
Given a subset $E\subset\mathbb R^d$ and $x\in\mathbb R^d$, we write $E+x=\{y+x:y \in E\}$ and define
$$
E_x=E\cap (E+x),
$$
and for a sequence $(x_n)_{n\ge 1}$ we define $E_{x_1,x_2}=(E_{x_1})_{x_2}=E\cap (E+x_1)\cap (E+x_2)\cap (E+x_1+x_2)$,
and recursively

$$
E_{x_1,\ldots, x_{n+1}}=(E_{x_1,\ldots,x_n})_{x_{n+1}}.
$$

  In what follows we will denote the Lebesgue measure on $\mathbb R^d$ by $\lambda$ and for $y\in\mathbb R^d$,
$|y|_2$ its $l^2$-norm.
 We will need the following  multi-point version of Steinhaus's Theorem (see \cite{St}) in $\mathbb R^d$.  Although we have not found this particular statement in the literature, we expect that it is well-known, so we omit the proof. 

\begin{lemma}
  \label{steinhaus} Let $E\subset \mathbb R^d$ with $\lambda(E)>0$.
  Then, for every $n\in \N$, there exists a $\delta=\delta(n,E)>0$ such that for all $y_1,\ldots,y_n\in\mathbb R^d$
  with $|y_i|_2<\delta$, $1\le i\le n$, the set

  $$
 E_{y_1,\ldots,y_n}
$$
is non-empty.
  \end{lemma}

  Note that
\[
v(\pv)
=
\dfrac{1-\bro}{1-\bro+\dfrac{2}{m}\sum\limits_{r=1}^{m}\sum\limits_{k=0}^{m-1}\prod\limits_{i=0}^{r-1}\rho_{i+k}}=
\dfrac{1-\gamma}{1+\gamma+\dfrac{2}{m}\sum\limits_{r=0}^{m-2}\sum\limits_{k=0}^{m-1}\prod\limits_{i=0}^{r}\rho_{i+k}}.
\]

\begin{proof}[Proof of Theorem \ref{thm:ae}]
The set of $\pv$ for which $\bro=0$ has Lebesgue measure 0, so we may assume that $\bro\ne 0$.  By symmetry (apply the result to $1-\pv$ when $\bro>1$) we may assume that $\bro<1$, so we can use the formula \eqref{awesome}.

The statement is trivial for $m=3$ since there is exactly 1 speed for each $\pv$ in this case.  We fix $m\ge 4$ in what follows.

Let $\mc{J}_m$ denote the set of permutations of $\{0,\dots, m-1\}$ that are not compositions of rotations and reversal.   To prove the theorem, it is sufficient to show that for Lebesgue a.e.~$\pv$ any permutation $\sigma\in \mc{J}_m$ does not give the same velocity, i.e.~$v(\pv)\ne v((\pv)_\sigma)$.  Given a permutation $\sigma$ of $\{0,\ldots,m-1\}$, for $i\ne j$, $0\le i,j\le m-1$, we will say that $\sigma(i)$ is
\emph{adjacent} to $\sigma(j)$ if $\sigma(i)=\sigma(j)+1$ or $\sigma(i)=\sigma(j)-1$, where the sum is $\mod m$.  Note that $\mc{J}_m$ is precisely the set of permutations that do not preserve all adjacency relations, i.e. $\sigma\in \mc{J}_m$ if and only if there exists a $k\in\{0,\ldots,m-1\}$ such that
$\sigma(k)$ is not adjacent to $\sigma(k+1)$.

\medskip

  \noindent {\it Step 1.}
Let $\sigma\in \mc{J}_m$.   It is enough to show that the set $E$ of $\bs{\rho}_m=(\rho_0,\dots, \rho_{m-1})\in (0,\infty)^m$ for which 
\begin{equation}
\label{a1}
\sum\limits_{r=0}^{m-2}\sum\limits_{k=0}^{m-1}\prod\limits_{i=0}^{r}\rho_{i+k}=\sum\limits_{r=0}^{m-2}\sum\limits_{k=0}^{m-1}\prod\limits_{i=0}^{r}\rho_{\sigma(i+k)}
\end{equation}
has Lebesgue measure 0. We will \emph{assume that $\lambda(E)>0$ 
} 
and obtain a contradiction.

\medskip

\noindent{\it Step 2.}
Note that the terms in \eqref{a1} with $r=0$ and $r=m-2$ cancel out, so we have that $(\rho_0,\ldots,\rho_{m-1})\in E$
if and only if
$$
H(\rho_0,\ldots,\rho_{m-1})=0,
$$
where
\begin{equation*}
H(x_0,\ldots,x_{m-1})
=\sum\limits_{r=1}^{m-3}\sum\limits_{s=0}^{m-1}\prod\limits_{i=0}^{}x_{i+s}-\sum\limits_{r=1}^{m-3}\sum\limits_{s=0}^{m-1}\prod\limits_{i=0}^{r}x_{\sigma(i+s)}
\end{equation*}

\medskip

\noindent {\it Step 3.} For each $0\le i\le m-1$, $h>0$ and function $g:\mathbb R^m\to\mathbb R$  define
\begin{eqnarray*}
  &\Delta^i_hg(x_0,\ldots,x_{m-1})\\
  &=\frac{1}{h}\left(g(x_0,\ldots,x_{i-1},x_i+h,x_{i+1},\ldots,x_{m-1})-g(x_0,\ldots,x_{m-1})\right).
\end{eqnarray*}
Note that the operator $\Delta^i_h$ is simply a discrete derivative.

Let us describe how iterations of these operators act on products of $\rho_i$, which is a central component of the proof.  Consider a function $G:\R^m \to \R$ of the form 
\begin{equation*}
G(x_0,\dots, x_{m-1})=\prod_{i \in A}x_i,
\end{equation*}
where $A\subset \{0,1,\dots, m-1\}$.  It is easy to see that
\begin{equation*}
\Delta_h^{j}G(x_0,\dots, x_{m-1})=\begin{cases}
\prod_{i \in A \setminus\{j\}}x_i, & \text{ if }j \in A\\
0, & \text{ otherwise}.
\end{cases}
\end{equation*}
It follows that if $\#A\le \ell$ then 
\begin{equation}
\label{heyhey}
\Delta_h^{j_\ell}\dots \Delta_h^{j_2}\Delta_h^{j_1}G(x_0,\dots, x_{m-1})=\begin{cases}
1& \text{ if }\{j_1,\dots, j_\ell\}=A\\
0 & \text{ otherwise}.
\end{cases}
\end{equation}

For $0\le j\le m-1$ and a permutation $\sigma'$ let 
$$
 H_{\sigma',j}(x_0,\ldots,x_{m-1})=\prod_{i=2}^{m-1} 
  x_{\sigma'(i+j)}.
  $$
Note that $x_{\sigma'(j)}$ and $x_{\sigma'(j+1)}$ are ``missing'' from this product. 

Since $\sigma \in \mc{J}_m$, there exists $k\in\{0,\ldots,m-1\}$ such that 
$\sigma(k)$ and $\sigma(k+1)$ are not adjacent.      It follows from \eqref{heyhey} that 
\begin{equation*}
\Delta^{\sigma(2+k)}_{h}\Delta^{\sigma(3+k)}_{h}\cdots \Delta^{\sigma(m-1+k)}_{h} H_{\sigma,j}(x_0,\ldots,x_{m-1}) 
=\delta_{j,k}.
\end{equation*}
Recall that $I$ is the identity permutation.  Then $H_{I,j}(x_0,\dots, x_{m-1})=\prod_{i=2}^{m-1} 
  x_{i+j}$ is missing $x_j$ and $x_{j+1}$, where $j$ and $j+1$ are adjacent.  The set $\{\sigma(i+k):i=2,\dots, m-1\}$ is missing $\sigma(k)$ and $\sigma(k+1)$ which are not adjacent.  It follows that $\{\sigma(k),\sigma(k+1)\}\ne \{j,j+1\}$ so by \eqref{heyhey},
\begin{equation*}
\Delta^{\sigma(2+k)}_{h}\Delta^{\sigma(3+k)}_{h}\cdots \Delta^{\sigma(m-1+k)}_{h} H_{I,j}(x_0,\ldots,x_{m-1}) 
=0.
\end{equation*}
There are $m-2$ discrete derivatives here, and in the definition of $H$, only $r=m-3$ gives a product of $m-2$ terms.  From  \eqref{heyhey} we see that
\begin{align}
&\Delta^{\sigma(2+k)}_{h}\cdots \Delta^{\sigma(m-1+k)}_{h}H(x_0,\ldots,x_{m-1})\label{heyhey2}\\
&=\Delta^{\sigma(2+k)}_{h}\cdots \Delta^{\sigma(m-1+k)}_{h}\Bigg[
\sum\limits_{s=0}^{m-1}\prod\limits_{i=0}^{m-3}x_{i+s}-\sum\limits_{s=0}^{m-1}\prod\limits_{i=0}^{m-3}x_{\sigma(i+s)}\Bigg].\nonumber
\end{align}
Using the substitution $j=s-2$ (mod $ m$) shows that the term in square brackets is 
\[\sum_{j=0}^{m-1}H_{I,j}(x_0,\dots, x_{m-1})-\sum_{j=0}^{m-1}H_{\sigma,j}(x_0,\dots, x_{m-1}),\]
and therefore \eqref{heyhey2} is equal to 
$0-1=-1$ for every $(x_0,\dots, x_{m-1})$.

Now, suppose that $\lambda(E)>0$.   From Lemma \ref{steinhaus} (with $d$ and $n$ therein both equal to $m$), there exists $\delta(m,E)>0$ such that for all $y_0,\dots, y_{m-1}\in \R^m$ with $|y_i|_2<\delta$, $0\le i \le m-1$, the set $E_{y_0,\dots, y_{m-1}}$ is non-empty.   Taking $y_i=h e_{i+1}$ for $h\in (0,\delta)$, where $e_1,\dots, e_{m}$ are the canonical basis vectors in $\R^m$, it follows that there exists a point $(\rho_0,\ldots,\rho_{m-1})\in E$ such that
$(\rho_0+h,\rho_1,\ldots,\rho_{m-1}), (\rho_0,\rho_1+h,\rho_2,\ldots,\rho_{m-1}), (\rho_0+h,\rho_1+h,\rho_2,\ldots,\rho_{m-1})$, etc., are all
in $E$ also. Let $B$ be this set of points.  Then 
\[B\subset \{(x_0,\dots x_{m-1}):x_i\in \{\rho_i,\rho_i+h\} \text{ for every }i=0,\dots, m-1\}.\]
By definition 
\begin{equation}
\Delta^{\sigma(2+k)}_{h}\Delta^{\sigma(3+k)}_{h}\cdots \Delta^{\sigma(m-1+k)}_{h}H(\rho_0,\ldots,\rho_{m-1}) \label{byebye}
\end{equation}
is a linear combination of terms of the form $H(x^i_0,\dots, x^i_{m-1})$ with each $\bs{x}^i\in B$.  But by Step 2 $H(\bs{x})=0$ for all $\bs{x}\in B$, so \eqref{byebye} is equal to 0, which contradicts the fact that \eqref{heyhey} is equal to 1 for all $\bs{x}$.
\end{proof}

\section{Proof of Theorem \ref{thm:greedy3}}
\label{sec:r3}
In this section we prove Theorem \ref{thm:greedy3}.  Recall that for $r\in [m]$ and a permutation $\sigma\in \mc{S}_m$ we have 
\[P_r(\sigma;\av):=\sum_{k=0}^{m-1}\prod_{i=0}^{r-1}a_{\sigma^{-1}(k+i)},\]
with indices interpreted $\Mod m$.   

Suppose that we prove the result for $r=k\in [m]$.  Since the entries of $\av$ are decreasing, the reciprocals $\av^{-1}$ of $\av$ listed in reverse order (write this vector as $\av^\dagger$) are also increasing.  So we know that the $\greedy$ maximises $P_k(\cdot,\av^\dagger)$.  But each $P_r$ is trivially invariant to reversals so $\greedy$ maximises $P_k(\cdot,\av^{-1})$.  The observation \eqref{trick1} then shows that $\greedy$ maximises $P_{m-k}(\cdot, \av)$.  It therefore suffices to prove the claim for $r=2,3$.  We prove each of these results by induction on $m$.


\begin{proof}[Proof for $r=2$.]
For the base case $m=2$ there is nothing to prove.  We will assume the result for $m$ and prove it for $m+1$.  Let $\bs{a}_{m+1}$ be such that $a_0\ge \dots\ge a_{m}$.  Write $\bs{a}_{m+1}=(\bs{a}_m,a_m)$.  Let $\sigma$ denote a permutation of $\{0,1,\dots, m\}$, and let $j_\sigma=\sigma(m)$.  We have that 
\begin{align*}
P_2(\sigma, \bs{a}_{m+1})=P_2( \sigma, (\av,a_m))=\sum_{i=0}^{m}a_{\sigma^{-1}(i)}a_{\sigma^{-1}(i+1)}=\sum_{i=0}^{m}b_ib_{i+1},
\end{align*}
where $b_i=a_{\sigma^{-1}(i)}$ (and $b_{m+1}=b_0$).  Note that $b_{j_\sigma}=a_m$.  Let $\hat{\sigma}$ denote the permutation of $\{0,1,\dots, m-1\}$ defined by 
\begin{align*}
\hat\sigma^{-1}(i)=\begin{cases}
\sigma^{-1}(i), & \text{ if }i<j_\sigma\\
\sigma^{-1}(i+1), & \text{ if }i\in [j_\sigma,m-1],
\end{cases}
\end{align*}
where if $j_\sigma=m$, the second situation doesn't arise.  Now note that
\begin{align*}
P_2( \sigma, (\av,a_m))&=P_2( \hat{\sigma}, \av)+R(\sigma,\bs{a}_{m+1}),
\end{align*}
where
\begin{align*}
R(\sigma,\bs{a}_{m+1})=b_{j_\sigma-1}a_m+b_{j_\sigma+1}a_m - b_{j_\sigma-1}b_{j_\sigma+1}.
\end{align*}
We claim the the greedy permutation $\greedy(m+1)$ on $\{0,1,\dots, m\}$ maximises  both $P_2( \hat{\sigma}, \av)$ and $R(\sigma,\bs{a}_{m+1})$, and hence it maximises $P_2( \sigma, (\av,a_m))$.

Note that $\hat{\sigma}_{\text{greedy}}(m+1)=\greedy(m)$. 
By the induction hypothesis, $\greedy(m+1)$ then maximises $P_2( \hat{\sigma}, \av)$.

Let $f(x,y)=xa_m+ya_m-xy$ be defined for all $x,y\ge a_m$, and note that $R(\sigma,\bs{a}_{m+1}) =f(b_{j_{\sigma}-1},b_{j_{\sigma}+1})$.  The partial derivatives are $f_1(x,y)=a_m-y\le 0$ and $f_2(x,y)=a_m-x\le 0$.  Therefore the largest possible value of $f(x,y)$ for $x=b_i$, $y=b_{i'}$ with $i\ne i'$ occurs with $\{b_i,b_{i'}\}=\{a_{m-1},a_{m-2}\}$.  In other words, any permutation $\sigma$ that puts $a_m$ between $a_{m-1}$ and $a_{m-2}$ maximises $R(\sigma,\bs{a}_{m+1})$.  Since $\greedy(m+1)$ has this property, this completes the proof for $r=2$.
\end{proof}

We now prove the result for $r=3$, using the same notation as above.
\begin{proof}[Proof for $r=3$.]  
 For $m=3$ there is nothing to prove.  Note that
\begin{align*}
P_3(\sigma,\bs{a}_{m+1})&=P_3(\hat{\sigma},\av)+R(\sigma,\bs{a}_{m+1}),
\end{align*}
where now
\begin{align*}
R(\sigma,\bs{a}_{m+1})&=b_{j_\sigma-2}b_{j_\sigma-1}a_m+b_{j_\sigma-1}b_{j_\sigma+1}a_m +b_{j_\sigma+1}b_{j_\sigma+2}a_m\\
&\quad - b_{j_\sigma-2}b_{j_\sigma-1}b_{j_\sigma+1}-b_{j_\sigma-1}b_{j_\sigma+1}b_{j_\sigma+2}.
\end{align*}
By the induction hypothesis,  the term $P_3(\hat{\sigma},\av)$ is maximised by any $\sigma$ such that $\hat{\sigma}=\greedy(m)$.  Note that $\sigma =\greedy(m+1)$ has this property.

For $x_{-2},x_{-1},x_{1},x_2\in [a_m,1)$ let 
\[f(x_{-2},x_{-1},x_1,x_2)=x_{-2}x_{-1}a_m+x_{-1}x_{1}a_m +x_{1}x_{2}a_m- x_{-2}x_{-1}x_{1}-x_{-1}x_{1}x_2.\]
Note that $R(\sigma,\bs{a}_{m+1})=f(b_{j_\sigma-2},b_{j_\sigma-1},
b_{j_\sigma+1},b_{j_\sigma+2})$.  
The partial derivatives are:
\begin{align*}
f_1(x_{-2},x_{-1},x_1,x_2)&=x_{-1}(a_m-x_1)\le 0\\
f_2(x_{-2},x_{-1},x_1,x_2)&=x_{-2}(a_m-x_1)+x_1(a_m-x_{-2})\le 0\\
f_3(x_{-2},x_{-1},x_1,x_2)&=x_{-1}(a_m-x_{2})+x_2(a_m-x_{-1})\le 0  \\
f_4(x_{-2},x_{-1},x_1,x_2)&=x_1(a_m-x_{-1})\le 0.
\end{align*}
The term $R(\sigma,\bs{a}_{m+1})$ is therefore maximised at some $\sigma$ for which $\{b_{j_\sigma-2},b_{j_\sigma-1},b_{j_\sigma+1},
b_{j_\sigma+2}\}=\{a_{m-4},a_{m-3},a_{m-2},a_{m-1}\}$.  Note that  $\greedy(m+1)$ also has this property.  We proceed assuming that $\{b_{j_\sigma-2},b_{j_\sigma-1},b_{j_\sigma+1},
b_{j_\sigma+2}\}=\{a_{m-4},a_{m-3},a_{m-2},a_{m-1}\}$, and we will show that $\greedy(m+1)$ maximises $R(\sigma,\bs{a}_{m+1})$ among all $\sigma$ for which $\{b_{j_\sigma-2},b_{j_\sigma-1},b_{j_\sigma+1},
b_{j_\sigma+2}\}=\{a_{m-4},a_{m-3},a_{m-2},a_{m-1}\}$.  This suffices to prove then that $\greedy(m+1)$ maximises $P_3(\sigma,\bs{a}_{m+1})$.

Now note that by adding and subtracting the terms $b_{j_\sigma+2}   b_{j_\sigma-2}b_{j_\sigma-1}$ and $b_{j_\sigma+1}b_{j_\sigma+2}b_{j_\sigma-2}$ we can write 
\begin{align*}
R(\sigma,\bs{a}_{m+1})&=P_3\Big(I_5,(b_{j_\sigma-2},b_{j_\sigma-1},a_m,b_{j_\sigma+1},
b_{j_\sigma+2})\Big)\\
&\quad-P_3\Big(I_4,(b_{j_\sigma-2},b_{j_\sigma-1},b_{j_\sigma+1},
b_{j_\sigma+2})\Big),
\end{align*}
where $I_k$ is the identity permutation on $k$ elements.  The first term on the right hand side is equal to 
\begin{align*}
&\prod_{i=-2}^2b_{j_\sigma-i} \times P_2\Big(I_5,(1/b_{j_\sigma-2},1/b_{j_\sigma-1},1/a_m,1/b_{j_\sigma+1},
1/b_{j_\sigma+2})\Big)\\
&=\prod_{i=0}^4 a_{m-i}\times P_2\Big(I_5,(1/b_{j_\sigma-2},1/b_{j_\sigma-1},1/a_m,1/b_{j_\sigma+1},
1/b_{j_\sigma+2})\Big).
\end{align*}
The product prefactor is constant.  By the result already established for $r=2$ and the symmetry of the greedy permutation, the quantity $P_2$ here is maximised (among those as above) by any permutation $\sigma$ for which the vector $(1/b_{j_\sigma-2},1/b_{j_\sigma-1},1/a_m,1/b_{j_\sigma+1},
1/b_{j_\sigma+2})$ is already the greedy ordering (or a symmetry of it) of $\{a_{m-4},\dots, a_{m}\}$.   Note that $\sigma=\greedy(m+1)$ has this property.

Finally, the term $P_3\Big(I_4,(b_{j_\sigma-2},b_{j_\sigma-1},b_{j_\sigma+1},
b_{j_\sigma+2})\Big)$ is equal to 
\[P_1\Big(I_4,(1/b_{j_\sigma-2},1/b_{j_\sigma-1},1/b_{j_\sigma+1},
1/b_{j_\sigma+2})\Big),\]
and since $P_1$ does not depend on the permutation, we have that $\greedy(m+1)$ is a minimiser of this term as well.  This completes the proof.
\end{proof}

\doubleblind{\subsection*{Acknowledgements}
MH thanks Robert Medland and Victor Kleptsyn for various helpful discussions near the beginning of this project.  MH was supported by  Future Fellowship FT160100166 from the Australian Research Council. AR was supported by Fondecyt 1180259
and Iniciativa Cient\'\i fica Milenio.
}

\bibliographystyle{plain}

\end{document}